\documentclass[reqno,a4paper,11pt]{amsart}
\usepackage{fullpage}
\usepackage{amsmath, amssymb}
\usepackage{mathrsfs}
\usepackage[matrix, arrow]{xy}
\usepackage{amsthm}
\usepackage{graphicx}

\newtheorem{dfn}{Definition}[section]
\newtheorem{thm}[dfn]{Theorem}
\newtheorem{lemma}[dfn]{Lemma}
\newtheorem{prop}[dfn]{Proposition}
\newtheorem{cor}[dfn]{Corollary}
\newtheorem{rmk}[dfn]{Remark}
\theoremstyle{definition}

\theoremstyle{definition}
\newtheorem{exmp}[dfn]{Example}
\newcommand{\notmid}{\mathrel{\ooalign{$\mkern-5mu\not$\crcr$|$}}}
\title{Motivic Modularity of CM K3 Surfaces}
\author{Rikuto Ito}
\address{Graduate School of Mathematics, Nagoya University, Furo-cho, Chikusa-ku, Nagoya, 464-8601, Japan}
\email{ito.rikuto.g7@s.mail.nagoya-u.ac.jp}
\begin{document}
\begin{abstract}
The main result of this paper is the explicit construction of  algebraic Hecke characters of K3 surfaces defined over a number field $k\subset \mathbb{C}$ with complex multiplication. In particular, we give the explicit determination of the algebraic Hecke characters of Kummer surfaces associated to simple abelian surfaces with complex multiplication. 
\end{abstract}
\maketitle
\tableofcontents

\section{Introduction}
 Let $A$ be an abelian variety of dimension $n$ with CM (complex multiplication) defined over a number field $k \subset \mathbb{C}$. Shimura and Taniyama \cite{shimura} showed that  there exist algebraic Hecke characters $\chi_{1}, \cdots, \chi_{2n}$ over $k$ such that 
\begin{align}
    L(s, A)=\prod^{2n}_{i=1}L(s, \chi_{i}), \tag{1.1}
\end{align}
where $L(s, A)$ is the Hasse-Weil $L$-function of $A$ and $L(s, \chi_{i})$ are the Hecke $L$-functions of $\chi_{i}$. This statement is a consequence of \textbf{the main theorem of complex multiplication} for  abelian varieties. 
\\$\indent$ Piatetski-Shapiro and Shafarevich \cite{ps} proved an analogue of (1.1) for \textbf{K3 surfaces with complex multiplication (CM K3 surfaces,  in short)}  using their Kuga-Satake constructions and Shimura-Taniyama's result. 
\\$\indent$ In some arithmetic problems, e.g., the calculation of Brauer groups, it is important to know the explicit description of the algebraic Hecke characters. The explicit descriptions of  Hecke characters associated to CM K3 surfaces have been obtained in the following cases:
\begin{itemize}
        
    \item{\textbf{K3 surfaces with non-symplectic group actions.} Livn\'e, Sch\"utt, and Yui \cite{yui} discussed the fields of definition and the Hecke characters of K3 surfaces with non-symplectic group actions. In such a class, K3 surfaces with the minimal Picard rank are of CM type (Example \ref{k3ex} (2)). They showed that the fields of definition of these K3 surfaces are cyclotomic fields over $\mathbb{Q}$ and the associated algebraic Hecke characters are given by \textbf{Jacobi sums}. }
    \item{\textbf{Kummer surfaces associated to the self-product of a CM elliptic curve.}} Deuring \cite{deu} constructed the algebraic Hecke characters for CM elliptic curves defined over number fields, and proved the equality (1.1) for $n=1$. Shioda and Inose \cite{si} showed that the Hecke characters of the Kummer surface associated to the self-product of a CM elliptic curve are given by the square of the Hecke character of the CM elliptic curve. 
    \item{\textbf{Kummer surfaces associated to the product of two CM elliptic curves.}} Rizov \cite{Rizov10} proved the main theorem of complex multiplication for the products of two elliptic curves. Using  the Kummer construction, one can compute the algebraic Hecke characters of the associated Kummer surfaces. 
    \item{\textbf{Certain K3 surfaces with Picard rank 16.}}
    Costa, Elsenhans, Jahnel, and  Voight \cite{costa} introduced  three explicit K3 surfaces  of Picard rank 16, realized  as double planes.  They conjectured that these surfaces have CM by CM fields  obtained by adjoining $\sqrt{-1}$ to certain totally real cubic fields.  Under the assumption  that the surfaces are of \textbf{CM type}, they proved that the 6-dimensional Galois representations associated with their transcendental parts are isomorphic to representations induced from algebraic Hecke characters over the corresponding CM fields. In particular, after restriction to these CM fields, the representations  decompose as  direct sums of algebraic Hecke characters and their conjugates. 
\end{itemize}
\begin{rmk}Note that K3 surfaces in the above  cases do not have the Picard rank 8 or 14 as shown in \cite[Theorem 1 and 2]{yui} and \cite[Corollary 2.3, Proposition 2.4, and Corollary 2.6]{hulek}. But Taelman \cite{tae} proved that  for any CM field $E$ over $\mathbb{Q}$ of degree at most 20, there exists a K3 surface over $\mathbb{C}$ with complex multiplication by $E$. The existence of CM K3 surfaces with Picard rank 8 or 14 follows from \cite{tae}. 
\end{rmk}
\noindent \indent The contribution of this paper is to provide a formula (Theorem \ref{main theorem}) for the algebraic Hecke characters of  CM K3 surfaces over number fields, including those with Picard rank 8 and 14. Moreover, this formula includes the following new examples:
\begin{itemize}
    \item {\textbf{Kummer surfaces associated to a simple CM abelian surface (Proposition \ref{onegai}).}} This result is independent of Shioda-Inose \cite{si} and Rizov \cite{Rizov10}. By applying the main theorem of complex multiplication for simple abelian surfaces, one can determine the Hecke characters of the associated Kummer surfaces by a method analogous to those of Shioda-Inose and Rizov, and show that they coincide with those given in Theorem \ref{main theorem}.
\end{itemize}
\begin{rmk}
    Although Theorem \ref{main theorem} does not prove that the three K3 surfaces in \cite{costa} are of CM type, it shows that, under this assumption, their transcendental Galois representations are governed by algebraic Hecke characters in the sense of this paper. In this way, their result \cite[Theorem 1.2(a)]{costa}are recovered by our general construction once the CM type assumption is imposed. 
\end{rmk}
\subsection{Main theorem}
$\indent$ Let $k$ be a number field, and let $\mathbb{A}^{\times}_{k}$ denote the idele group of $k$. Fix a rational prime $l$ and an isomorphism of fields $\overline{\mathbb{Q}_{l}}\cong \mathbb{C}$. By class field theory, there exists a one-to-one correspondence (see Theorem $\ref{1gal}$):
\begin{align}
\{ 1\textendash{\rm dimensional ~locally ~algebraic~ representation}\}&\leftrightarrow \{ {\rm algebraic~Hecke~character}\}\notag
\\(\rho_{l, \chi} : {\rm Gal}(\overline{k}/k)\to\overline{\mathbb{Q}_{l}}^{\times})&\leftrightarrow  (\chi : \mathbb{A}^{\times}_{k}\to\mathbb{C}^{\times}).    \notag
\end{align}
An $n$-dimensional semisimple $l$-adic abelian  representation $\rho_{l}$ is said to be {\bf modular} if there exist algebraic Hecke characters $\chi_{1}, \cdots, \chi_{n}$ such that $\rho_{l}\cong \rho_{l, \chi_{1}}\oplus\cdots\oplus \rho_{l, \chi_{n}}$.
\\$\indent$ Let $k$ be a number field and let $\overline{k}$ be an algebraic closure of $k$. Let $X$ be a K3 surface over $k$, i.e.,  a smooth projective surface over $k$ such that $H^{1}(X, \mathscr{O}_{X})=0$ and $\Omega_{X/k}^{2}\cong \mathscr{O}_{X}$. Since the Galois group ${\rm Gal}(\overline{k}/k)$ acts on $X_{\overline{k}}:=X\times_{{\rm Spec}(k)}{\rm Spec}(\overline{k})$, we obtain the 22-dimensional $l$-adic Galois representation 
\begin{align}
\rho_{l} : {\rm Gal}(\overline{k}/k)\to GL(H^{2}(X_{\overline{k}}, \mathbb{Q}_{l}(1))). \tag{1.2}
\end{align}
Fix an embedding $\overline{k}\hookrightarrow \mathbb{C}$. Then there are isomorphisms
\begin{align}
    H^{2}(X_{\overline{k}}, \mathbb{Q}_{l}(1))\cong H^{2}(X_{\mathbb{C}}, \mathbb{Q}(1))\otimes \mathbb{Q}_{l}\cong ({\rm NS}(X_{\mathbb{C}})\otimes\mathbb{Q}_{l})\oplus (T(X_{\mathbb{C}})\otimes \mathbb{Q}_{l}), \notag 
\end{align}where $T(X_{\mathbb{C}}) := {\rm NS}(X_{\mathbb{C}})^{\bot} \subset H^{2}(X_{\mathbb{C}}, \mathbb{Z}(1))$. Since ${\rm NS}(X_{\mathbb{C}})\cong {\rm NS}(X_{\overline{k}})$, the group ${\rm Gal}(\overline{k}/k)$ acts on ${\rm NS}(X_{\mathbb{C}})\otimes \mathbb{Q}_{l}$. Hence we obtain a $(22-\rho(X_{\overline{k}}))$-dimensional Galois representation 
\begin{align}
    {\rm Gal}(\overline{k}/k) \to GL(T_{l}(X_{\overline{k}})\otimes \mathbb{Q}_{l}), \tag{1.3} 
\end{align}where $T_{l}(X_{\overline{k}}) := T(X_{\mathbb{C}})\otimes\mathbb{Z}_{l}$. The  $\mathbb{Z}_{l}$-module $T_{l}(X_{\overline{k}})$ is called the \textbf{$l$-adic transcendental lattice} of $X$.
 \\\\ $\indent$ We obtain the following new result.
    \begin{thm}\label{main theorem} 
        Let $X$ be a K3 surface defined over a number field $k\subset \mathbb{C}$. Suppose that $X$ has CM over $k$ by a  CM field $E$, i.e., $X_{\mathbb{C}}$ is of CM type and $k\supset E$. Then there exists a unique locally constant homomorphism $u : \mathbb{A}^{\times}_{k} \to E^{\times}\subset \mathbb{A}^{\times}_{E}$ satisfying the following properties (i)-(iii):
        \\ \indent (i) Fix an embedding $\tau : E\to \mathbb{C}$. The function \begin{align}
            \chi_{\tau} : \mathbb{A}^{\times}_{k}\xrightarrow{u(-)\cdot \overline{{\rm Nm}_{k/E}(-)}/{\rm Nm}_{k/E}(-)} \mathbb{A}^{\times}_{E}\xrightarrow{\tau-{\rm projection}} \mathbb{C}^{\times}\notag\end{align}
         is an algebraic Hecke character, where ${\rm Nm}_{k/E} : \mathbb{A}^{\times}_{k}\to \mathbb{A}^{\times}_{E}$ denotes the norm map for the extension $k/E$, and $\overline{(-)} : \mathbb{A}^{\times}_{E}\to \mathbb{A}^{\times}_{E}$ denotes the complex conjugation.
        \\\indent (ii) Let \begin{align}\rho_{l} : G_{k}\to GL(T_{l}(X_{\overline{k}})\otimes \mathbb{Q}_{l}) \notag
        \end{align}
    be the $(22-\rho(X_{\overline{k}}))$-dimensional $l$-adic representation. Then $\rho_{l}$ is diagonalized by the algebraic Hecke characters $\chi_{\tau}$, where $\tau$ runs over ${\rm Hom}(E, \mathbb{C})$.
    \\ \indent (iii) Fix an embedding $\tau\in {\rm Hom}(E, \mathbb{C})$ and a prime ideal $\mathfrak{p}\subset k$ with $(l, \mathfrak{p})=1$. Then the following conditions are equivalent: 
     \\(a) The $l$-adic representation $\rho_{l}$ is unramified at $\mathfrak{p}$, i.e., if $I_{\mathfrak{p}}$ denotes the inertia group of $\mathfrak{p}$, then $\rho_{l}(I_{\mathfrak{p}})=\{id\}$.
     \\(b) The Hecke character $\chi_{\tau}$ is unramified at $\mathfrak{p}$, i.e., if $\mathscr{O}_{\mathfrak{p}}$ is the ring of integers of $k_{\mathfrak{p}}$, then  $\chi_{\tau}(\mathscr{O}^{\times}_{\mathfrak{p}})=1.$
       \end{thm}
       \begin{rmk}
           If $k$ does not contain $E$, Theorem \ref{main theorem} can be extended to a  K3 surface $X$  over  $kE$, and  Theorem \ref{main theorem} is  valid for   K3 surface $X$ over $kE$.
       \end{rmk}
    By Theorem \ref{main theorem}, we obtain the following equality 
   \begin{align}
       L(s, T_{l}(X_{\overline{k}})\otimes \mathbb{Q}_{l})=\prod_{\tau\in {\rm Hom}(E, \mathbb{C})}L(s, \chi_{\tau}). \tag{1.4}
   \end{align}
   Here $L(s, T_{l}(X_{\overline{k}})\otimes \mathbb{Q}_{l})$ denotes the $L$-function of $\rho_{l}$ (see \cite[Section 2]{tay}). We call the equality (1.4)  the \textbf{Motivic modularity}  of $X$.
   \\\\ \indent A key point in the proof of Theorem \ref{main theorem} is to construct algebraic Hecke characters $\chi_{\tau}$ satisfying  condition (i) of Theorem \ref{main theorem}. The fact that the characters $\chi_{\tau}$ satisfy the property (ii) follows from the semisimplicity of $\rho_{l}$. These are applications  of the main theorem of complex multiplication for K3 surfaces by Rizov \cite{Rizov10}. 
   \\\indent On the other hand, the semisimplicity also follows from the theory of absolute Hodge cycles by Deligne \cite{dmos};  Indeed, let $X$ be a K3 surface with CM over $k$ by a CM field $E$. The Hodge cycles 
   \begin{align}E:={\rm End}_{Hdg}T(X_{\mathbb{C}})_{\mathbb{Q}}\subset H^{4}(X_{\mathbb{C}}\times X_{\mathbb{C}}, \mathbb{Q})(2) \notag 
   \end{align}are absolute Hodge cycles (as a consequence of Deligne's Principle B \cite[I, Theorem 2.12]{dmos} and the Kuga-Satake construction \cite[Proposition 6.5]{del 1}). Hence the action of ${\rm Gal}(\overline{k}/k)$ on $T(X_{\mathbb{C}})_{\mathbb{A}_{f}}$ commutes with the action  of $E$. It follows that the image of the Galois representation $\rho:{\rm Gal}(\overline{k}/k)\to GL(T(X_{\mathbb{C}})_{\mathbb{A}_{f}})$ is contained in  $\mathbb{A}^{\times}_{E, f}$. Since $T(X_{\mathbb{C}})_{\mathbb{A}_{f}}$ is a rank-1 $\mathbb{A}_{E, f}$-module, it follows that $\rho$ is semisimple. 
   \begin{rmk}
       In general, Zarhin and Skorobogatov \cite{sz} proved that the Galois representation (1.2) is semisimple for any K3 surface. Their result is based on the semisimplicity of the Galois representations for abelian varieties due to Faltings \cite{fal}. 
       \end{rmk}
 In the final section, we compute, as an example of Theorem \ref{main theorem}, the algebraic Hecke characters in Theorem \ref{main theorem} for Kummer surfaces associated with CM simple abelian surfaces. The main theorem of complex multiplication for simple abelian surfaces describes the Galois action on the \textbf{Tate module} (Milne \cite{milne2}). Using this result, we compute the Galois action on the \textbf{transcendental part} of a CM simple abelian surface (see Proposition \ref{simplecm}).  Finally, we show that the Hecke characters obtained from Proposition \ref{simplecm} coincide with those given in Theorem \ref{main theorem}.    
\\\\$\textbf{Acknowledgements}$
\\$\indent$I would like to thank the referee for critical and constructive comments and suggestions for the improvement of the paper. 
  \\$\indent$This paper is based on the author's master's thesis \cite{riku}. The author would like to thank his advisor, Sho Tanimoto, for his support and advice. He kindly supported the author when the author was taking a leave of absence. The author thanks  Noriko Yui for her comments and suggestions. The author  would also like to thank Professors Kazuhiro Fujiwara, Kazuhiro Ito, Takuya Yamauchi, and Alexei N Skorobogatov for their comments and suggestions.
\\ $\indent$The author was partially supported by JST FOREST program Grant number JPMJFR212Z.
\section{Preliminaries}
\subsection{Abelian $l$-adic representations}
$\indent$For a rational prime $l$, let $\mathbb{Q}_{l}$ be the $l$-adic field, and let  $\overline{\mathbb{Q}_{l}}$ be an algebraic closure of $\mathbb{Q}_{l}$.
\begin{dfn} (Abelian $l$-adic representations)
    \\$\indent$ Let $k$ be a number field. An \textbf{$l$-adic Galois representation} is a finite-dimensional $\overline{\mathbb{Q}_{l}}$-continuous representation of the Galois group ${\rm Gal}(\overline{k}/k)$. An \textbf{abelian $l$-adic representation} is an $l$-adic Galois representation whose image is commutative.
\end{dfn}
\begin{exmp}
    Let $k$ be a number field, let $\overline{k}$ be an algebraic closure, and let $X$ be a smooth projective variety defined over $k$. Let $X_{\overline{k}}:=X\times_{{\rm Spec}(k)}{\rm Spec}(\overline{k})$. Then the $i$-th  \'etale cohomology $H^{i}(X_{\overline{k}}, \mathbb{Q}_{l})$ is a finite-dimensional $\mathbb{Q}_{l}$-vector space. If $\sigma\in{\rm Gal}(\overline{k}/k)$, it induces the automorphism of $X_{\overline{k}}$. Hence we obtain the $l$-adic Galois representation 
\begin{align} \rho_{l} : {\rm Gal}(\overline{k}/k)\to GL(H^{i}(X_{\overline{k}}, \mathbb{Q}_{l})). \tag{2.1}\end{align}
\end{exmp}
 Now we recall the main statements of global class field theory.
    Let $k$ be a number field, let $\mathbb{A}^{\times}_{k}$ be the idele group, and let $k^{ab}$ be the  maximal abelian extension of $k$.  Let $[-, k] : \mathbb{A}^{\times}_{k}\to {\rm Gal}(k^{ab}/k)$ be the reciprocity map for $k$. The \textbf{Artin map} ${\rm art}_{k} : \mathbb{A}^{\times}_{k}\to {\rm Gal}(k^{ab}/k)$ is defined by ${\rm art}_{k} := [-, k]^{-1}$.
    \begin{thm}\cite[Chapter II, Theorem 3.5]{silv}\label{class 2}
        Let $k$ be a number field. The Artin map ${\rm art}_{k}$ has the following properties :
        \\\indent (i) The Artin map is surjective and $k^{\times}$ is contained in the kernel of ${\rm art}_{k}$.
        \\\indent (ii) For any finite abelian extension $K/k$, the Artin map is compatible with the norm map ${\rm Nm}_{K/k} : \mathbb{A}^{\times}_{K}\to \mathbb{A}^{\times}_{k}$, i.e., for all $x\in \mathbb{A}^{\times}_{K}$ the following diagram
        \[\xymatrix{\mathbb{A}^{\times}_{K}
        \ar[rr]^{{\rm art}_{K}(x)}\ar[d]_{{\rm Nm}_{K/k}}&&{\rm Gal}(K^{ab}/K)\ar[d]\\
        \mathbb{A}^{\times}_{k} \ar[rr]_{{\rm art}_{k}({\rm Nm}_{K/k}(x))} && {\rm Gal}(k^{ab}/k)}
        \] is commutative.
        \\\indent (iii) Let $\mathfrak{p}$ be a prime ideal of $k$, let $I^{ab}_{\mathfrak{p}}\subset {\rm Gal}(k^{ab}/k)$ be an inertia group of $\mathfrak{p}$ for the extension $k^{ab}/k$, let $\pi_{\mathfrak{p}}$ be a uniformizer at $\mathfrak{p}$, and let $K/k$ be any abelian extension that is unramified at $\mathfrak{p}$. Then
            ${\rm art}_{k}(\pi_{\mathfrak{p}})^{-1}|_{K}$
        is the Frobenius for $K/k$ at $\mathfrak{p}$ and ${\rm art}_{k}(\mathscr{O}_{\mathfrak{p}}^{\times})=I^{ab}_{\mathfrak{p}}$ where $\mathscr{O}_{\mathfrak{p}}$ is the  ring of integers of $k_{\mathfrak{p}}$. 
    \end{thm}
Let $k$ be a number field, and  denote by  $\Sigma_{k}$ the set of all finite places of $k$. Let $v\in \sum_{k}$ be a finite place of $k$, and let $\mathfrak{p}$ be the prime ideal corresponding to $v$. Put $q :=|\mathscr{O}_{k}/\mathfrak{p}|$. Then there exists the following exact sequence:
\begin{align}
    1\to I_{\mathfrak{p}}\to {\rm Gal}(\overline{k_{\mathfrak{p}}}/k_{\mathfrak{p}})\to {\rm Gal}(\overline{\mathbb{F}_{q}}/\mathbb{F}_{q})\to 1,\tag{2.2}
\end{align}
where $I_{\mathfrak{p}}$ is the inertia group at $\mathfrak{p}$. The group ${\rm Gal}(\overline{\mathbb{F}_{q}}/\mathbb{F}_{q})$ is topologically generated by the \textbf{geometric Frobenius} $x\mapsto x^{q^{-1}}$. By the exact sequence (2.2), there exists an element ${\rm Frob}_{\mathfrak{p}}\in {\rm Gal}(\overline{k}_{\mathfrak{p}}/k_{\mathfrak{p}})$ which induces the geometric Frobenius on $\overline{\mathbb{F}_{q}}$, up to the action of $I_{\mathfrak{p}}$. This  element ${\rm Frob}_{\mathfrak{p}}$ is called the $\textbf{(geometric) Frobenius element}$ at $\mathfrak{p}$. 
\begin{dfn}
    Let $v\in \sum_{k}$ and let $\mathfrak{p}$ be the prime ideal associated to $v$. An $l$-adic representation $\rho_{l}$ is said to be \textbf{unramified} at $v$ if $\rho_{l}(I_{\mathfrak{p}})=\{id\}$.
\end{dfn}
\begin{prop}
    Let $X$ be a smooth projective surface defined over a number field $k$. Suppose that the $l$-adic representation $\rho_{l}$ is abelian, i.e., the image is abelian. Then $\rho_{l}$ induces a Galois representation
    \begin{align}
        \rho_{l}^{ab} : {\rm Gal}(k^{ab}/k)\to GL(H^{2}_{et}(X_{\overline{k}}, \mathbb{Q}_{l})).\notag
    \end{align}
    Moreover, $\rho_{l}$ is unramified at $v$ if and only if $\rho_{l}^{ab}(I^{ab}_{\mathfrak{p}})=\{id\}$ where $\mathfrak{p}$ is the prime of $k$ associated to $v$.
\end{prop}
\begin{proof}
    Let $q$ be the norm of $\mathfrak{p}$. Since  finite extensions of $\mathbb{F}_{q}$ are cyclic, we have $\overline{\mathbb{F}_{q}}=\mathbb{F}^{ab}_{q}$. Consider the following commutative diagram:
    \[\xymatrix{I_{\mathfrak{p}}\ar[r] \ar[d]&{\rm Gal}(\overline{k_{\mathfrak{p}}}/k_{\mathfrak{p}})\ar[r] \ar[d]_{res} & {\rm Gal}(\overline{\mathbb{F}_{q}}/\mathbb{F}_{q})\ar[d]^{id}
    \\ I_{\mathfrak{p}}^{ab}\ar[r] &{\rm Gal}(k^{ab}_{\mathfrak{p}}/k_{\mathfrak{p}})\ar[r] &  {\rm Gal}(\mathbb{F}^{ab}_{q}/\mathbb{F}_{q}).}\]
    Suppose that $\rho_{l}(I_{\mathfrak{p}})=\{id\}$. Let $\sigma \in I_{\mathfrak{p}}^{ab}$. Then there exists an element $\sigma^{'}\in {\rm Gal}(\overline{k}/k)$ such that $\sigma^{'}|_{k^{ab}}=\sigma$. This implies that  $\rho_{l}^{ab}(\sigma)=\rho_{l}(\sigma^{'})$. On the other hand, since $I^{ab}_{\mathfrak{p}}$ is the kernel of ${\rm Gal}(k^{ab}_{\mathfrak{p}}/k_{\mathfrak{p}})\to {\rm Gal}(\mathbb{F}_{q}^{ab}/\mathbb{F}_{q})$, it follows from the above diagram that $\sigma^{'}\in I_{\mathfrak{p}}$. Hence we have $\rho_{l}^{ab}(\sigma)=\rho_{l}(\sigma^{'})=id$.
    \\ The converse is clear.
\end{proof}
 If an $l$-adic representation $\rho_{l} : {\rm Gal}(\overline{k}/k)\to GL(V)$ is unramified at $v$, then the action of $\rho_{l}({\rm Frob}_{\mathfrak{p}})$ on $V$ is well-defined. Then we denote the characteristic polynomial of $\rho_{l}({\rm Frob}_{\mathfrak{p}})$ by $P_{v, \rho_{l}}(T)$.
\begin{dfn} (Rational representations)
   \\\indent An $l$-adic representation $\rho_{l}$ is said to be \textbf{rational} if there exists a finite subset $S\subset \sum_{k}$ such that:
   \\\indent (i) Any element of $\sum_{k}\backslash  S$ is unramified for  $\rho_{l}$.
   \\\indent (ii) If $v\notin S$, the coefficients of $P_{v, \rho_{l}}(T)$ lie in  $\mathbb{Q}$.
\end{dfn}
\begin{rmk}
    Dwork \cite{dw} proved that the Galois representation (2.1) is rational. This is a part of the Weil conjecture.
\end{rmk}
\begin{dfn} (Hecke characters)
\\$\indent$Let $k$ be a number field. 
    A continuous homomorphism $\chi : \mathbb{A}^{\times}_{k}\to\mathbb{C}^{\times}$ is called a \textbf{Hecke character} if $\chi$ is trivial on $k^{\times}$.
\end{dfn}
\begin{dfn} Let $k$ be a number field. Let $v\in \sum_{k}$ be a finite place,  and let $\mathfrak{p}$ be the prime ideal associated with $v$. A Hecke character $\chi : \mathbb{A}^{\times}_{k}\to \mathbb{C}^{\times}$ is said to be \textbf{unramified} at $v$ if it is trivial on $\mathscr{O}_{\mathfrak{p}}^{\times}\subset \mathbb{A}^{\times}_{k}$.
\end{dfn}
We recall that there exists a correspondence between the set of one-dimensional locally algebraic Galois representations and  the set of algebraic Hecke characters. Let $k$ be a number field,  and let $v$ be an infinite place.  We denote by $k^{\times 0}_{v}$ the connected component of the identity of $k_{v}$.
\begin{dfn} (Algebraic Hecke characters)
\\$\indent$ Let $k$ be a number field.  A Hecke character $\chi : \mathbb{A}^{\times}_{k}\to \mathbb{C}^{\times}$ is said to be \textbf{algebraic} if, for any infinite place $v$, there exist integers $(a_{\tau})_{\tau \in {\rm Hom}_{\mathbb{R}}(k_{v}, \mathbb{C})}$ such that 
  \begin{align}
      \chi_{v}(x_{v})=\prod_{\tau\in {\rm Hom}_{\mathbb{R}}(k_{v}, \mathbb{C})}\tau(x_{v})^{-a_{\tau}},~~~(x_{v}\in k_{v}^{\times 0}), \notag
  \end{align}
  where $\chi_{v} : k_{v}^{\times}\to \mathbb{C}^{\times}$ is the $v$-component of $\chi$.
  \end{dfn}
  \begin{dfn}(Locally algebraic representations)
  \\\indent Let $k$ be a number field, and let $\rho_{l} : {\rm Gal}(k^{ab}/k)\to GL(V)$ be an $l$-adic abelian representation. 
  \\\indent (i)~Let $v|l$. Put $\rho_{v} := \rho_{l}|_{{\rm Gal}(\overline{k_{v}}/k_{v})}$, and let  $T:={\rm Res}_{k_{v}/\mathbb{Q}_{l}}(\mathbb{G}_{m, k_{v}})$. The representation $\rho_{l}$ is said to be \textbf{locally algebraic at $v$} if there exists an algebraic morphism $r : T\to GL(V)$ such that the following diagram commutes on a neighborhood of 1 in $k^{\times}_{v}$:
  \[\xymatrix{{\rm Gal}(k_{v}^{ab})\ar[r]^{\rho_{v}}& GL(V) \\ k_{v}^{\times}.\ar[u]^{{\rm art}_{k}}\ar[ur]_{r}}\]
   \indent (ii) The $l$-adic abelian representation $\rho_{l}$ is said to be \textbf{locally algebraic} if it is locally algebraic at all places $v|l$.
\end{dfn}
\begin{thm}\cite[Chapter III-1-Proposition 2]{serre}\label{1gal}
    Let $l$ be a prime. Fix an isomorphism of fields $\iota : \overline{\mathbb{Q}_{l}}\to \mathbb{C}$. Let $k$ be a number field. Then there exists a one-to-one correspondence  
    \begin{align}
     \{1\textendash{\rm dimensional~ locally~algebraic~representation}\}&\leftrightarrow\{{\rm algebraic~Hecke~character}\}\notag
     \\ \rho_{l}&\leftrightarrow\chi.\notag   
    \end{align}
    Under this correspondence, we have $\chi_{v}=\iota\circ \rho_{l, v}\circ {\rm art}_{k}$ for all finite places $v\notmid l$.
\end{thm}
\subsection{Mumford-Tate groups}
$\indent$Deligne gave a representation-theoretic definition and characterization of  Hodge structures in \cite{dmos}. We now  recall them. Let $V$ be a rational Hodge structure of weight $n$. In other words, there exists a decomposition $V_{\mathbb{C}}=\bigoplus_{p+q=n}V^{p, q}$ with $\overline{V^{p, q}}=V^{q, p}$. This gives a real representation 
\begin{align}
    \rho : \mathbb{C}^{\times}\to GL(V_{\mathbb{R}}), z\mapsto \rho(z) : v\mapsto (z^{-p}\overline{z}^{-q})v\notag
\end{align}
for $v\in V^{p, q}$. In particular, for $t\in \mathbb{R}^{\times}$, we have $\rho(t)=t^{-n} : v \mapsto t^{-n}v$. Conversely, if $V$ is a finite-dimensional $\mathbb{Q}$-vector space and $\rho : \mathbb{C}^{\times}\to GL(V_{\mathbb{R}})$ is a real representation such that $\rho(t)=t^{-n}$ for all $t\in \mathbb{R}^{\times}$, then  the $\mathbb{C}$-subspaces
\begin{align}
    V^{p, q}:=\{v\in V_{\mathbb{C}} | ~\rho(z)(v)=z^{-p}\overline{z}^{-q}v, {\rm for~all}~z\in \mathbb{C}^{\times}\},~~~p+q=n\notag
\end{align}
define a rational Hodge structure  of weight $n$ on $V$. Hence we obtain a one-to-one correspondence:
\begin{align}
    \{{\rm rational~Hodge~structure~of~weight}~n\} \leftrightarrow \{{\rm real~representation}~\mathbb{C}^{\times}\to GL(V_{\mathbb{R}})\notag
    \\ {\rm such~that} ~t\mapsto t^{-n}\cdot(-)~{\rm on}~\mathbb{R}^{\times}\}.\notag
\end{align}
On the other hand, a real representation $\mathbb{C}^{\times}\to GL(V_{\mathbb{R}})$ can be regarded as a representation of the Deligne torus $\mathbb{S}:={\rm Res}_{\mathbb{C}/\mathbb{R}}(\mathbb{G}_{m, \mathbb{C}})$. This is an $\mathbb{R}$ algebraic group defined by 
\begin{align}
    \{\mathbb{R}\textendash{\rm algebras}\}\ni A \mapsto \mathbb{S}(A):=(A\otimes_{\mathbb{R}}\mathbb{C})^{\times}\in \{{\rm groups}\}.\notag 
\end{align}
Hence, a rational Hodge structure on $V$ of weight $n$ corresponds to an algebraic $\mathbb{R}$-morphism $\rho : \mathbb{S}\to GL(V_{\mathbb{R}})$ such that $\rho(t)=t^{-n}$ for all $t\in \mathbb{R}^{\times}$. 
\begin{dfn} (Mumford-Tate groups)
   \\$\indent$ Let $V$ be a rational Hodge structure, and let $\rho : \mathbb{S} \to GL(V_{\mathbb{R}})$ be the algebraic $\mathbb{R}$-representation associated with $V$. The \textbf{Mumford-Tate group} ${\rm MT}(V)$ of $V$ is the smallest $\mathbb{Q}$-algebraic subgroup of $GL(V)$ such that $\rho(\mathbb{S}(\mathbb{R}))\subset {\rm MT}(V)(\mathbb{R})$. 
\end{dfn}
\begin{exmp}
    (i) Let $X$ be an algebraic  K3 surface over $\mathbb{C}$. Consider the second  cohomology $H^{2}(X, \mathbb{Z})$. This is a Hodge structure of weight $2$, hence the Tate twist $H^{2}(X, \mathbb{Z})(1)$ is of weight 0. By the Lefschetz $(1, 1)$-theorem (see  \cite[Chapter III, Proposition 3.3.2]{Huy2}), the N\'eron-Severi group ${\rm NS}(X)$ is isomorphic to $H^{2}(X, \mathbb{Z})(1)\cap H^{1,1}(X)$. Hence the orthogonal complement $T(X)$ is a sub-Hodge structure of weight $0$.
    \\\indent (ii) Let $A$ be an abelian surface over $\mathbb{C}$. Then the second cohomology $H^{2}(A, \mathbb{Z})$ is a Hodge structure of weight $2$, hence the Tate twist $H^{2}(A, \mathbb{Z})(1)$ is of weight 0. By the Lefschetz $(1, 1)$-theorem, ${\rm NS}(A)\cong H^{2}(A, \mathbb{Z})(1)\cap H^{1, 1}(A)$. Hence the orthogonal complement $T(A)$ is a sub-Hodge structure of weight $0$. 
    \end{exmp}
    ~\\
\subsection{K3 surfaces with CM}
~
\\$\indent$   To explain CM, we first define CM fields. 
   \begin{dfn} (CM fields) Let $K$ be a number field.
       \\\indent (i) $K$ is said to be \textbf{totally real} if $\sigma(K)\subset \mathbb{R}$ for all inclusions $\sigma : K\to \mathbb{C}$.
       \\\indent (ii) $K$ is said to be \textbf{totally  imaginary} if $\sigma(K)\not\subset\mathbb{R}$ for all inclusions $\sigma : K\to \mathbb{C}$.
       \\\indent (iii) A \textbf{CM field} is a totally imaginary quadratic extension of a totally real field.
   \end{dfn}
   \begin{exmp}
      (1) An imaginary quadratic field is a CM field. In fact, $\mathbb{Q}$ is totally real, and the imaginary quadratic extension  is not contained in $\mathbb{R}$.
      \\\indent (2) Let $p$ be a prime $\neq 2$, and let $\zeta_{p}$ be a $p$-th root of unity. Consider the field extensions $\mathbb{Q}\subset \mathbb{Q}(\zeta_{p}+\zeta_{p}^{-1})\subset \mathbb{Q}(\zeta_{p})$. Since $\zeta_{p}^{2}-(\zeta_{p}+\zeta_{p}^{-1})\zeta_{p}+1=0$, the extension $\mathbb{Q}(\zeta_{p}+\zeta_{p}^{-1})\subset \mathbb{Q}(\zeta_{p})$ has degree $2$. If $\sigma : \mathbb{Q}(\zeta_{p})\to \mathbb{C}$ is an inclusion, there exists an integer $n$ with $\sigma(\zeta_{p})=\zeta_{p}^{n}, (n, p)=1$. This shows that $\mathbb{Q}(\zeta_{p})$ is totally imaginary. Finally, the field $\mathbb{Q}(\zeta_{p}+\zeta_{p}^{-1})$ is totally real because 
      \begin{align}
          \sigma(\zeta_{p}+\zeta_{p}^{-1})=\zeta_{p}^{n}+\zeta_{p}^{-n}=2\cos{\left(\frac{2n\pi}{p}\right)}\in\mathbb{R}.\notag
      \end{align}
      Hence $\mathbb{Q}(\zeta_{p})$ is a CM field.
   \end{exmp}
   Let $X$ be an algebraic K3 surface over $\mathbb{C}$. Let $E_{X}:={\rm End}_{{\rm Hdg}}(T(X)_{\mathbb{Q}})$. If $\omega$ is a nowhere vanishing holomorphic $2$-form on $X$, then we have a homomorphism of $\mathbb{Q}$-algebras: 
    \begin{align}
        \sigma_{X} : E_{X}\to {\rm End}_{\mathbb{C}} (\mathbb{C}\omega)\cong \mathbb{C}. \tag{2.3}
    \end{align}
    Note that this morphism is independent of the choice of $\omega$. 
    \begin{prop}
        Let $X$ be a complex algebraic K3 surface. Then the homomorphism $\sigma_{X}$ of (2.3) is injective, and $E_{X}$ is a number field. 
    \end{prop}
    \begin{proof}
        See  \cite[Chapter III, Corollary 3.6]{Huy}.
    \end{proof}
    \begin{dfn} (CM K3 surfaces over $\mathbb{C}$)
    \\$\indent$ Let $X$ be an algebraic K3 surface over $\mathbb{C}$. We say that $X$ has \textbf{CM} over $\mathbb{C}$ if the following conditions hold: 
        \\\indent (i) $E_{X}$ is a CM field.
        \\\indent (ii) ${\rm dim}_{E_{X}}T(X)_{\mathbb{Q}}=1.$
    \end{dfn}
  By condition (ii), we have $[E_{X} : \mathbb{Q}]={\rm dim}_{\mathbb{Q}}T(X)_{\mathbb{Q}}.$
  \begin{rmk}
      Valloni \cite{val21} proved that the \textbf{reflex field} of the rational Hodge structure $T(X)_{\mathbb{Q}}$ is $\sigma_{X}(E_{X})$. 
  \end{rmk}
    
  \begin{exmp}\label{k3ex}
      (1) If $X$ is a singular K3 surface, i.e., it has  the maximal Picard rank $20$, then $X$ has CM over $\mathbb{C}$. Then the reflex field is an imaginary quadratic field. These considerations are given in \cite[Chapter 3, Remark 3.10]{Huy}. 
      \\\indent (2)  Let $X$ be a K3 surface over $\mathbb{C}$. Suppose that there exists an automorphism $\sigma\in {\rm Aut}(X)$ of order $N$ such that the transcendental rank ${\rm rank}_{\mathbb{Z}} T(X)$ is equal to  $\varphi(N)$,  where $\varphi$ is Euler's function. Then $X$ has CM over $\mathbb{C}$, and the reflex field is the cyclotomic field $\mathbb{Q}(\zeta_{N})$, where $\zeta_{N}$ is a primitive $N$-th root of unity.
      \\\indent (3) Let $A$ be an abelian surface over $\mathbb{C}$, and let $X$ be the Kummer surface associated to $A$. Then  $X$ has CM over $\mathbb{C}$ if and only if $A$ has CM over $\mathbb{C}$ (see \cite[Lemma 9.3]{Ito}).  
      \end{exmp}
\begin{dfn}\label{def field}
    Let $X$ be a K3 surface over a number field $k\subset \mathbb{C}$. Then we say that $X$ has CM over $k$ if 
    \\\indent (i) the scalar extension $X_{\mathbb{C}} := X\times_{{\rm Spec}(k)}{\rm Spec}(\mathbb{C})$ has CM over $\mathbb{C}$,
    \\\indent (ii) $k\supset \sigma_{X_{\mathbb{C}}}(E_{X_{\mathbb{C}}})$.
\end{dfn} 
\begin{rmk}Mukai, Nikulin,  and Buskin(\cite{mukai}, \cite{nik}, and \cite{bus}) proved that the Hodge conjecture for the self-product of a CM K3 surface is true (see  \cite[Corollary of Theorem 1.1]{bus}). Using this, Valloni \cite{val21} gave a condition for $X$ to have complex multiplication over $k$, analogous to the corresponding condition for abelian varieties.
\end{rmk}
Finally, we introduce a fundamental fact for CM K3 surfaces. Let $X$ be a K3 surface with CM over $\mathbb{C}$, and  let $G:={\rm MT}(T(X)_{\mathbb{Q}})$. Zarhin  proved the following  (see  \cite[Theorem 2.3.1]{Zarhin}): 
    \begin{thm}
       The algebraic group $G$ is isomorphic,  as a $\mathbb{Q}$-algebraic group, to the kernel of
        \begin{align}
            {\rm Nm} : {\rm Res}_{E/\mathbb{Q}}\mathbb{G}_{m, E}\to {\rm Res}_{F/\mathbb{Q}}\mathbb{G}_{m, F},\notag
        \end{align}
    where $E=\sigma_{X}(E_{X})$ and $F$ is the maximal totally real subfield of $E$.
    \end{thm}
    Hence the rational points of $G$ are given by 
    \begin{align}
        G(\mathbb{Q})=\{ x\in E^{\times} | x\overline{x}=1\},\notag
    \end{align}
    where $\overline{(-)}$ denotes the complex conjugation $E\to E$.
    \section{The proof of Theorem \ref{main theorem}}
    \subsection{The main theorem of complex multiplication by Rizov}
    \indent Suppose that $X$ has CM over a number field $k$. Let $\mathbb{A}_{f}: = (\prod_{l}\mathbb{Z}_{l})\otimes_{\mathbb{Z}}\mathbb{Q}$ be the ring of finite adeles of $\mathbb{Q}$. Fix an embedding $\overline{k}\to \mathbb{C}$. Then we have an isomorphism
    \begin{align}
        H^{2}(X_{\overline{k}}, \mathbb{Z}_{l}(1))\cong H^{2}(X_{\mathbb{C}}, \mathbb{Z}(1))\otimes_{\mathbb{Z}}\mathbb{Z}_{l}. \notag
    \end{align}
    Hence
    \begin{align}
        (\prod_{l}H^{2}(X_{\overline{k}}, \mathbb{Z}_{l}(1))\otimes_{\mathbb{Z}}\mathbb{Q}\cong H^{2}(X_{\mathbb{C}}, \mathbb{Z}(1))\otimes_{\mathbb{Z}} \mathbb{A}_{f},\notag
    \end{align}
and
\begin{align}
    (\prod_{l}T_{l}(X_{\overline{k}}))\otimes_{\mathbb{Z}}\mathbb{Q}\cong T(X_{\mathbb{C}})\otimes_{\mathbb{Z}} \mathbb{A}_{f}. \notag
\end{align}
    Thus we obtain a Galois representation $\rho : {\rm Gal}(\overline{k}/k)\to GL(T(X_{\mathbb{C}})_{\mathbb{Q}})(\mathbb{A}_{f})$ whose image is contained in $ G(\mathbb{A}_{f})$
   \\
   \indent To construct Hecke characters $\mathbb{A}_{k}^{\times}\to \mathbb{C}^{\times}$, we first recall the main theorem of CM for $X_{\mathbb{C}}$ by Rizov\cite{Rizov10}. This gives a description of $\rho$ via class field theory. Let $E$ be the reflex field of $X_{\mathbb{C}}$. Since $X$ has CM over $k$, we have $E\subset k$ by definition $\ref{def field}$. Since $E$ is a CM field, complex conjugation $\overline{(-)}$ is an automorphism of $E$.
Hence it induces an involution $\overline{(-)} : \mathbb{A}^{\times}_{E}\to \mathbb{A}^{\times}_{E}$.
\\ $\indent$The main theorem of complex multiplication was proved by Rizov \cite{Rizov10}. The following statement follows from  \cite[Theorem 2.1]{Ito}. The proof is given  in \cite[Corollary 4.4]{mp}.
\begin{thm}\label{risov} The action of ${\rm Gal}(\overline{k}/k)$ on $T(X_{\mathbb{C}})\otimes\mathbb{A}_{f}$ factors through a continuous homomorphism
\begin{align}
   \rho^{ab} : {\rm Gal}(k^{ab}/k)\to G(\mathbb{A}_{f})=\{x\in \mathbb{A}^{\times}_{E, f} | x\overline{x}=1\},\notag
\end{align}
and the following diagram commutes:
\[\xymatrix{\mathbb{A}^{\times}_{k}\ar[rr]^{{\rm art}_{k}}\ar[d]_{{\rm Nm}_{k/E}}&& {\rm Gal}(k^{ab}/k)\ar[rr]^{\rho^{ab}} && G(\mathbb{A}_{f})\ar[d]^{quotient} \\ \mathbb{A}^{\times}_{E}\ar[rr]_{proj}&& \mathbb{A}^{\times}_{E, f}\ar[rr]_{\overline{(-)}/(-)}&& G(\mathbb{A}_{f})/G(\mathbb{Q}).}
\]
\end{thm}
\begin{rmk}\label{ssrho}
    By Theorem \ref{risov}, the image of $\rho$ is contained in $\mathbb{A}_{E, f}^{\times}$. Since $T(X_{\mathbb{C}})_{\mathbb{A}_{f}}$ is a free $\mathbb{A}_{E, f}$-module of rank one, it follows that $\rho$ is semisimple.
\end{rmk}
\begin{rmk}(Valloni \cite{val22})\label{val}
      \\$\indent$In the situation of Theorem \ref{risov}, let $t\in \mathbb{A}^{\times}_{k}$  and put $s:={\rm Nm}_{k/E}(t)\in \mathbb{A}^{\times}_{E, f}.$ Then there exists a unique $u\in G(\mathbb{Q})$ such that 
      \begin{align}
          \rho^{ab}({\rm art}_k(t))=u\cdot \overline{s}/s\in G(\mathbb{A}_{f}).\notag
      \end{align}
  \end{rmk}
The element $\overline{s}/s$ has the following property:
  \begin{prop}\label{lcomp} Let $v$ be a finite place of $k$, and let $t\in {\rm Im}(\mathscr{O}_{{p}_{v}}^{\times}\to \mathbb{A}^{\times}_{k})$. Put $s := {\rm Nm}_{k/E}(t)$. Under the isomorphism $\prod_{l}T_{l}(X_{\overline{k}})\otimes \mathbb{Q}\cong T(X_{\mathbb{C}})\otimes \mathbb{A}_{f}$, for $(l, \mathfrak{p}_{v})=1$, the $l$-th component of $ \overline{s}/s\in GL(T(X_{\mathbb{C}}))(\mathbb{A}_{f})$ is equal to $1$.
  \end{prop} 
  \begin{proof}
  The $l$-th component of $\overline{s}/s$ is given by $\prod_{w|l}(\overline{s}/s)_{w}$. On the other hand, by the definition of the  norm, the $w$-th component of $s\in \mathbb{A}^{\times}_{E,f}$ is given by 
     \begin{align}
         \prod_{v^{'}|w}N_{k_{v^{'}}/E_{w}}(t_{v^{'}}). \notag
     \end{align}
    Fix $w|l$. If  $v|w$, then $v|l$, which is a contradiction. Hence, by assumption, we have $s_{w}=1$ . Similarly, $\overline{s}_{w}=1$. Since $w$ was arbitrary finite place of $E$ with $w|l$, it follows that $(\overline{s}/s)_{l}=1.$
  \end{proof}
    \subsection{Construction of Hecke characters}
  $\indent$  The goal of this section is to prove Theorem $\ref{main theorem}$. First, we prove the following lemma from algebraic number theory.
    \begin{lemma}
        Let $E$ be a CM field. Then the group
        \begin{align}G(\mathbb{Q})=\{ x\in E^{\times} | ~x\overline{x}=1\}\notag
        \end{align}is a discrete subgroup of $\mathbb{A}^{\times}_{E, f}$
    \end{lemma}
    \begin{proof}
        Let $F$ be the maximal totally real field of $E$. By Dirichlet's unit theorem, the group of units $\mathscr{O}_{F}^{\times}$ has a finite index in $\mathscr{O}_{E}^{\times}$ since $E$ is a CM field. Moreover, under the natural inclusion $E\to \mathbb{A}^{\times}_{E, f}$, we have 
        \begin{align}
            \mathscr{O}_{E}=E\cap (\prod_{v}\mathscr{O}_{v}).\notag
        \end{align}
        Since $\prod_{v}\mathscr{O}_{v}$ is open in $\mathbb{A}^{\times}_{E, f}$, it follows that  $\mathscr{O}_{E}$ is  open in $E$. Hence  $\mathscr{O}_{F}$ is also open in $E$. On the other hand, $G(\mathbb{Q})\cap \mathscr{O}_{F}=\{\pm 1\}$. Therefore, since $\mathbb{A}^{\times}_{E, f}$ is Hausdorff, it follows that $G(\mathbb{Q})$ is a discrete subgroup.
    \end{proof}
    \begin{prop}\label{hecke1}
        Suppose that $X$ is a K3 surface with CM over a number field $k\subset \mathbb{C}$. Let $E$ be the reflex field. Then there exists a unique locally constant homomorphism $ u: \mathbb{A}^{\times}_{k}\to E^{\times}$ satisfying the following condition: 
        \\ For all $t\in \mathbb{A}^{\times}_{k}$, 
        \begin{align}
            \rho^{ab}({\rm art}_k(t))=u(t)\overline{{\rm Nm}_{k/E}(t)}/{\rm Nm}_{k/E}(t)\in G(\mathbb{A}_{f})\notag.
        \end{align}
    \end{prop}
    \begin{proof}
    Under the notation of Remark $\ref{val}$, the correspondence $t\to u$ defines a map from $\mathbb{A}^{\times}_{k}$ to $G(\mathbb{Q})=\{ x\in E^{\times} | ~x\overline{x}=1\}$, which we  denote by $u$. The function $u$ is continuous, since it is defined by a functional equation involving the nonzero continuous functions $\rho^{ab}$ and $ {\rm Nm}_{k/E}$. Next, we show that $u$ is a group homomorphism. Let $t_{1}, t_{2}\in \mathbb{A}^{\times}_{k}$. By Theorem $\ref{risov}$, we have
    \begin{align}
        \rho^{ab}({\rm art}_{k}(t_{1}){\rm art}_{k}(t_{2}))&=\rho^{ab}({\rm art}_{k}(t_{1}))\rho^{ab}({\rm art}_{k}(t_{2}))\notag
        \\&=u(t_{1})\overline{{\rm Nm}_{k/E}(t_{1})}{\rm Nm}_{k/E}(t_{1})^{-1}u(t_{2})\overline{{\rm Nm}_{k/E}(t_{2})}{\rm Nm}_{k/E}(t_{2})^{-1}\notag
        \\&=u(t_{1})u(t_{2})\overline{{\rm Nm}_{k/E}(t_{1}t_{2})}{\rm Nm}_{k/E}(t_{1}t_{2})^{-1}. \notag\end{align}
       \indent On the other hand, since the Artin map ${\rm art}_{k} : \mathbb{A}^{\times}_{k}\to {\rm Gal}(k^{ab}/k)$ is a homomorphism, we have ${\rm art}_{k}(t_{1}){\rm art}_{k}(t_{2})={\rm art}_{k}(t_{1}t_{2})$. Hence, by uniqueness, we obtain $u(t_{1}t_{2})=u(t_{1})u(t_{2})$. Therefore, $u$ is a homomorphism from $\mathbb{A}^{\times}_{k}$ to $G(\mathbb{Q})$.
    \end{proof}
    Let $\tau : E\to \mathbb{C}$ be an inclusion. We define a character $\chi_{\tau} : \mathbb{A}_{k}^{\times}\to \mathbb{C}^{\times}$ by
    \begin{align}
        \chi_{\tau}(t):=\left(u(t)\cdot \overline{{\rm Nm}_{k/E}(t)}/{\rm Nm}_{k/E}(t)\right)_{\tau},\notag
    \end{align}
    where $(-)_{\tau}$ denotes the $\tau$-component in $\mathbb{A}^{\times}_{E}$.
\begin{thm}
(i) The characters $\chi_{\tau}$ are algebraic Hecke characters of $k$.
\\(ii)
     Let $l$ be a prime, and let $\rho_{l} : {\rm Gal}(\overline{k}/k)\to GL(T_{l}(X_{\overline{k}})\otimes\mathbb{Q}_{l}))$ be the l-adic Galois representation. Fix a prime ideal $\mathfrak{p}\subset k$ such that $(l, \mathfrak{p})=1$. Then the following conditions are equivalent: 
     \\(a) The $l$-adic representation $\rho_{l}$ is unramified at $\mathfrak{p}$.
     \\(b) The Hecke character $\chi_{\tau}$ is unramified at $\mathfrak{p}$, i.e., if $\mathscr{O}_{\mathfrak{p}}$ is the ring of integers of $k_{\mathfrak{p}}$, then $\chi_{\tau}(\mathscr{O}^{\times}_{\mathfrak{p}})=1.$
    \end{thm}
    \begin{proof}
(i): Let $t\in k^{\times}$. Since $k^{\times}$ is contained in the kernel of the Artin map by Theorem \ref{class 2}(i), we have ${\rm art}_{k}(t)=id_{k^{ab}}$. Since $t\in k$, we have  ${\rm Nm}_{k/E}(t)=N_{k/E}(t)$. Hence we obtain 
        \begin{align}
            \rho(id_{\overline{k}})=u(t)\cdot \overline{N_{k/E}(t)}N_{k/E}(t)^{-1}.\notag
        \end{align}
         Thus we have $\chi_{\tau}(t)=1$. This shows that $\chi_{\tau}$ is a Hecke character.
\\ (ii) By the Artin map, $\mathscr{O}^{\times}_{\mathfrak{p}}$ corresponds to $I_{\mathfrak{p}}$ (Theorem \ref{class 2}(iii)). Let $t\in {\rm Im}( \mathscr{O}^{\times}_{\mathfrak{p}}\to \mathbb{A}^{\times}_{k})$.
      \\ $(a)\Rightarrow (b) : $ Suppose that $\rho_{l}$ is unramified at $\mathfrak{p}$. By Theorem $\ref{risov}$, we obtain  
      \begin{align}
          \rho^{ab}({\rm art}_{k}(t))=u(t)\cdot \overline{{\rm Nm}_{k/E}(t)}{\rm Nm}_{k/E}(t)^{-1}.\notag
      \end{align}
      By assumption, $\rho^{ab}({\rm art}_{k}(t))$ acts trivially on the $l$-adic transcendental lattice $T_{l}(X_{\overline{k}}).$ On the other hand, by Proposition $\ref{lcomp}$, $\overline{{\rm Nm}_{k/E}(t)}/{\rm Nm}_{k/E}(t)$ also acts trivially. Hence $u(t)=1.$ By the definition of $t$, we have 
      \begin{align}\notag(\overline{{\rm Nm}_{k/E}(t)}/{\rm Nm}_{k/E}(t))_{\tau}=1,
      \end{align}
      and hence $\chi_{\tau}(t)=1.$
      \\ $(b)\Rightarrow (a)$ 
      Conversely, suppose that $\chi_{\tau}(t)=1.$ Then  $u(t)=1$. Let $x\in I_{\mathfrak{p}}$ be the image of $t$ under the Artin map. It suffices to show that $\rho_{l}(x)$ acts trivially on $T_{l}(X_{\overline{k}})$. This follows  from Proposition $\ref{lcomp}$.
    \end{proof}
    \begin{prop}\label{eigen}Let $\mathfrak{p}$ be a prime ideal with $(l, \mathfrak{p})=1$, and let $\pi_{\mathfrak{p}}\in \mathscr{O}_{\frak{p}}$ be a uniformizer. Suppose that $\rho_{l}$ is unramified at $\mathfrak{p}$. Then the action of $u(\pi_{\mathfrak{p}}^{-1})$ on $T_{l}(X_{\overline{k}})$ coincides with the Frobenius action $\rho_{l}({\rm Frob}_{\mathfrak{p}}).$ 
  \end{prop}
  \begin{proof}
      The Artin map ${\rm art}_{k}$ sends $\pi_{\mathfrak{p}}^{-1}$ to ${\rm Frob}_{\mathfrak{p}}$. Hence we have 
      \begin{align}
          \rho({\rm Frob}_{\mathfrak{p}})=u(\pi_{\mathfrak{p}}^{-1})\overline{{\rm Nm}_{k/E}(\pi_{\mathfrak{p}}^{-1})}/{\rm Nm}_{k/E}(\pi_{\mathfrak{p}}^{-1}).\notag
      \end{align}
      Taking the $l$-component, we obtain 
      \begin{align}
          \rho_{l}({\rm Frob}_{\mathfrak{p}})=u(\pi_{\mathfrak{p}}^{-1})\notag.
      \end{align}
  \end{proof}
  \begin{cor}\label{eigen2}
  Suppose that $(l, \mathfrak{p})=1$ and that $\rho_{l}$ is unramified at $\mathfrak{p}$. Then the eigenvalues of ${\rm Frob}_{\mathfrak{p}}$ are given by
      \begin{align}
      \{\chi_{\tau}(\pi_{\mathfrak{p}}^{-1})~ |~ \tau : E\to \mathbb{C}~~~inclusions~~~\}.\notag
      \end{align}
      In particular, the Galois representation $\rho_{l}$ is rational.
  \end{cor}
  \begin{proof}
      Since $X$ has CM over $k$, we have 
      \begin{align}
    {\rm rank}_{\mathbb{Z}}T(X_{\mathbb{C}})=[E:\mathbb{Q}]. \notag
      \end{align}
      Hence the characteristic polynomial of ${\rm Frob}_{\mathfrak{p}}$ has  degree $[E : \mathbb{Q}].$ By Proposition  $\ref{eigen}$, the set 
      \begin{align}
          \{u(\pi_{\frak{p}}^{-1})^{\tau}~|~\tau : E\to \mathbb{C} ~~~inclusions~~~\}\notag
      \end{align}
      gives all  the eigenvalues of ${\rm Frob}_{\mathfrak{p}}.$ On the other hand, we have 
      \begin{align}(\overline{{\rm Nm}_{k/E}(\pi_{\mathfrak{p}}^{-1}})/{\rm Nm}_{k/E}(\pi_{\mathfrak{p}}^{-1}))_{\tau}=1.\notag
      \end{align}
      Hence this set coincides with 
      \begin{align}
          \{ \chi_{\tau}(\pi_{\mathfrak{p}}^{-1})~|~ \tau:E\to \mathbb{C}~~~\}\notag.
      \end{align}
  \end{proof}
  \begin{rmk}
      The rationality is due to Dwork \cite{dw}. However, we prove this  for CM K3 surfaces without Dwork's result.
  \end{rmk}
  \begin{cor}\label{semisimpli}For each $\tau : E \to \mathbb{C}$, let $\rho_{l, \tau}$ be the one-dimensional $l$-adic representation associated with the Hecke character $\chi_{\tau}$.  Then we have
      \begin{align}
          \rho_{l}\cong \bigoplus_{\tau : E \to \mathbb{C}}\rho_{l, \tau}.\notag
      \end{align}
  \end{cor}
  \begin{proof}
      This follows from Corollary $\ref{eigen2}$, Remark \ref{ssrho},  Chebotarev's density theorem.
  \end{proof}
  
      By Corollary $\ref{semisimpli}$, we obtain Theorem \ref{main theorem}.
      \section{Examples; Kummer Surfaces associated to simple CM abelian surfaces}We now specialize Theorem \ref{main theorem} to Kummer surfaces associated to simple CM abelian surfaces. Using the classical CM theory of abelian varieties, we compare the Hecke characters obtained from Theorem \ref{main theorem} with those arising from the transcendental part of the abelian surface.
      \\\\\indent Following Chai-Conrad-Oort \cite{cco}, we recall  the definition of CM abelian varieties. Let $K$ be a \textbf{CM algebra}, i.e., a finite product of CM fields.  If $K=\prod K_{i}$, then we have ${\rm Hom}(K, \mathbb{C})=\coprod_{i} {\rm Hom}(K_{i}, \mathbb{C})$. In particular,  when ${\rm dim}_{\mathbb{Q}}K=2g$, the set ${\rm Hom}(K, \mathbb{C})$ consists of exactly $2g$-elements. A \textbf{type} $\Phi$ of $K$ is a subset $\Phi\subset {\rm Hom}(K, \mathbb{C})$ such that $\Phi\sqcup \overline{\Phi}={\rm Hom}(K, \mathbb{C})$, where $\overline{(-)}$ denotes complex conjugation of $\mathbb{C}$. The pair $(K, \Phi)$ is called a \textbf{CM type}. If $(K, \Phi)$ is a CM type of $K$, then the $\mathbb{C}$-vector space $\bigoplus_{\sigma\in\Phi}\mathbb{C}$ has a canonical structure of a $K$-vector space via the diagonal map $K\ni x\mapsto (\sigma(x))_{\sigma\in\Phi}\in \bigoplus_{\sigma\in\Phi}\mathbb{C}$.
      \begin{dfn}(CM abelian varieties) 
          Let $K$ be a CM-algebra with ${\rm dim}_{\mathbb{Q}} K=g$, let $(K, \Phi)$ be a CM type, and  let $A$ be a $g$-dimensional abelian variety over a subfield $k\subset \mathbb{C}$. Suppose that an embedding $i: K\to {\rm End}(A)\otimes_{\mathbb{Z}}{\mathbb{Q}}$ is given, and   there is an isomorphism $T_{0}A\otimes_{k}\mathbb{C}\cong \bigoplus_{\sigma\in \Phi}\mathbb{C}$ as $K\otimes_{\mathbb{Q}}\mathbb{C}$-modules, where $T_{0}A$ is the tangent space of $A$ at 0. The pair $(A, i)$ is said to be  a \textbf{CM abelian variety} over $k$ of CM type $(K, \Phi)$.
      \end{dfn}
      If $(A, i)$ is a CM abelian variety over $k$ of type $(K, \Phi)$, then $k$ contains the reflex field 
      \begin{align}\notag K^{\ast}:=\mathbb{Q}(\sum_{\sigma\in \Phi}\sigma(x)| x\in K).
      \end{align}Indeed, by the definition of a CM abelian variety, there is a diagonal action $(\sigma(x))_{\sigma\in \Phi}$ on the $k$-vector space $T_{0}A$ for every $x\in K$. Hence the trace $\sum_{\sigma\in \Phi}\sigma(x)$ lies in $k$.    
      \begin{prop}
          Let $(A, i)$ be a CM abelian surface over $k\subset \mathbb{C}$ of CM type $(K, \Phi)$. Then the Kummer surface $X:={\rm Km}(A)$ has CM over $k$ by the reflex field $K^{\ast}\subset\mathbb{C}$. 
      \end{prop} 
      \begin{proof}Let $\Phi=\{\sigma_{1}, \sigma_{2}\}.$ By the construction of Kummer surfaces, there is a $\mathbb{Q}$-Hodge isometry (up to the double twist)
          \begin{align}
              H^{2}(X_{\mathbb{C}}, \mathbb{Q}(1))\cong H^{2}(A_{\mathbb{C}}, \mathbb{Q}(1))\oplus \bigoplus^{16}_{i=1}[E_{i}], \notag 
          \end{align}
          where $[E_{i}]$ are the cohomology classes of the exceptional divisors of the minimal resolution $X\to A/\langle-1\rangle$.
          In particular, we obtain a Hodge isometry (up to the double twist) $T(X_\mathbb{C})_{\mathbb{Q}}\cong T(A_{\mathbb{C}})_{\mathbb{Q}}$.  Thus it  suffices to prove that the reflex field of $T(A)_{\mathbb{Q}}$ is $K^{\ast}$.
          \\\indent Via the composition
          \begin{align}\label{action}
            K\xrightarrow{i} {\rm End}(A)\otimes \mathbb{Q}\to{\rm End}_{Hdg}T(A_{\mathbb{C}})_{\mathbb{Q}}\to {\rm End}_{\mathbb{C}} H^{2, 0}(A_{\mathbb{C}})=\mathbb{C},  
         \tag{4.1} \end{align}
the field $K$ acts on $H^{2, 0}(A_{\mathbb{C}})$. 
          For $x\in K$, the element  $i(x)$ acts on $H^{2, 0}(A)$ as $\sigma_{1}(x)\sigma_{2}(x)$. In fact, since $T_{0}(A_{\mathbb{C}})\cong T_{0}A\otimes_{k}\mathbb{C}\cong \bigoplus_{\sigma\in \Phi} \mathbb{C}$,  the eigenvalues of $i(x)$ on $T_{0}(A_{\mathbb{C}})$ are  $\sigma_{1}(x)$ and $\sigma_{2}(x)$. Hence the eigenvalue of $i(x)$ on $H^{2, 0}(A_{\mathbb{C}})=\wedge^{2}H^{1, 0}(A_{\mathbb{C}})$ is $\sigma_{1}(x)\sigma_2(x)$. Thus the image of the composition (\ref{action}) is $K^{'}:=\mathbb{Q}(\sigma_{1}(x)\sigma_{2}(x)|~x\in K)\subset \mathbb{C}$. On the other hand, $K^{'}$ coincides with the reflex field $K^{\ast}$. Indeed, we have the relations
          \begin{align}
              \sigma_{1}(x)+\sigma_{2}(x)= \sigma_{1}(x+1)\sigma_{2}(x+1)-\sigma_{1}(x)\sigma_{2}(x)-1,\notag 
              \\ 2\sigma_{1}(x)\sigma_{2}(x)=(\sigma_{1}(x)+\sigma_{2}(x))^{2}-(\sigma_{1}(x^{2})+\sigma_{2}(x^{2})),\notag 
          \end{align}
         which show that $K^{'}=K^{\ast}$. 
      \\\indent Let $E(T(A_{\mathbb{C}})_{\mathbb{Q}})$ denote the reflex field of $T(A_{\mathbb{C}})_{\mathbb{Q}}$. To show  $K^{\ast}=E(T(A)_{\mathbb{Q}})$, we consider the following three cases.
         \\\indent (i) If $A_{\mathbb{C}}$ is simple, then the ring ${\rm End}(A_{\mathbb{C}})_{\mathbb{Q}}$ is a 4-dimensional CM field (see \cite[1.3.6.4. Proposition]{cco}). Hence $K$ is a CM-field. By (\ref{action}), the degree $[K^{\ast}: \mathbb{Q}]$ is equal to 4. Since the Picard number of $A_{\mathbb{C}}$ is $2$ (see \cite[Lemma 3.3]{murty}), we have $[E(T(A_{\mathbb{C}})_{\mathbb{Q}}): \mathbb{Q}]={\rm dim}_{\mathbb{Q}}T(A_{\mathbb{C}})_{\mathbb{Q}}=4$, and hence $K^{\ast}=E(T(A)_{\mathbb{Q}})$. 
       \\\indent (ii) If $A_{\mathbb{C}}$ is isogenous to a product of elliptic curves $E_{1}$ and $ E_{2}$, and if  $E_{1}$ and $E_{2}$ are not isogenous, then we have a canonical isomorphism of rings
       \begin{align}\label{nonis}
           {\rm End}(E_{1})_{\mathbb{Q}}\times {\rm End}(E_{2})_{\mathbb{Q}}\xrightarrow{\sim} {\rm End}(A_{\mathbb{C}})_{\mathbb{Q}}. \tag{4.2}
       \end{align} given by  $(f_{1}, f_{2})\mapsto f_{1}\times f_{2}$. 
       In fact, let $\varphi(x, y)=(f_{1}(x, y), f_{2}(x, y))\in {\rm End}(A_{\mathbb{C}})_{\mathbb{Q}}.$  Then  
       \begin{align}
           f_{i}(x, y)=f_{i}(x, 0)+f_{i}(0, y). \notag 
       \end{align}
       In particular, $f_{1}(0, y)$ is a homomorphism $E_{2}\to E_{1}$. Since $E_{2}$ and $E_{1}$ are not isogenous, it follows that $f_{1}(0, y)=0$ for all $y\in E_{2}$. By the same argument,  $f_{2}(x, 0)=0$ for all $x\in E_{1}$.  Hence the morphism (\ref{nonis}) is an isomorphism. 
      \\\indent In particular, the CM algebra $K$ is of the form  $K_{1}\times K_{2}$, where $K_{i}$ are imaginary quadratic fields that are not isomorphic. Hence the reflex field $K^{\ast}$ is generated by two non-isomorphic imaginary quadratic fields, and thus $[K^{\ast}:\mathbb{Q}]=4$. On the other hand, since  ${\rm dim}_{\mathbb{Q}}T(A)_{\mathbb{Q}}=4$ (see \cite[Corollary 2.3]{hulek}), we obtain  $[E(T(A)_{\mathbb{Q}}):\mathbb{Q}]=4$.  Hence $K^{\ast}=E(T(A_{\mathbb{C}})_{\mathbb{Q}})$. 
       \\\indent (iii) Suppose that $A_{\mathbb{C}}$ is isogenous to a product of two elliptic curves $E_{1}$ and $E_{2}$ which are isogenous. Then $K$ is a product of two isomorphic imaginary quadratic fields. In particular, the reflex field $K^{\ast}$ is a 2-dimensional vector space over $\mathbb{Q}$. Since the Picard number of $A_{\mathbb{C}}$ is $4$ (see \cite[Corollary 2.6]{hulek}), we have $[E(T(A_{\mathbb{C}})_{\mathbb{Q}}): \mathbb{Q}]=2$. Thus  $K^{\ast}=E(T(A)_{\mathbb{Q}})$. 
      \end{proof}
  As a consequence of Milne's complex multiplication theory \cite[Theorem 11.2]{milne2} for abelian varieties,  we will prove the main theorem of complex multiplication for the transcendental parts of simple abelian surfaces of CM type (Proposition \ref{simplecm}).  
  \\\\\indent Let $(A, i)$ be a CM abelian surface over $\mathbb{C}$ of CM type $(K, \Phi)$. Via the embedding  $i: K\hookrightarrow {\rm End}(A)\otimes \mathbb{Q}$, we see that the multiplicative group $K^{\times}\subset GL(H^{1}(A, \mathbb{Q}))$ is a maximal torus. Hence the Mumford-Tate group of $T(A)$ is contained in $K^{\times}$. In particular, the Hodge structure $h: \mathbb{C}^{\times}\to GL(H^{1}(A, \mathbb{R}))$ factors through the torus $(K\otimes_{\mathbb{Q}}\mathbb{R})^{\times}$.  We denote by $\mu_{h}$ the composition \begin{align} \mathbb{C}^{\times}\ni t\mapsto (t, 1)\in \mathbb{C}^{\times}\times \mathbb{C}^{\times}=\mathbb{S}(\mathbb{C})\xrightarrow{h_{\mathbb{C}}} (K\otimes_{\mathbb{Q}}\mathbb{C})^{\times}=\prod_{\varphi\in \Phi}\mathbb{C}^{\times}\times \prod_{\varphi\in \overline{\Phi}}\mathbb{C}^{\times}\notag\end{align}  By the definition of $h$, we have $\mu_{h}(t)=(t^{-1}, t^{-1}, 1, 1)\in \prod_{\varphi\in \Phi}\mathbb{C}^{\times}\times \prod_{\varphi\in \overline{\Phi}}\mathbb{C}^{\times}$.
      Moreover, the minimal field of definition of $\mu_{h}$ is the reflex field $K^{\ast}$, since  $\sigma \Phi=\Phi$ for every $\sigma\in {\rm Aut}(\mathbb{C}/K^{\ast})$. 
      \begin{prop}\label{rh}Let $(A, i)$ be a CM abelian surface over $\mathbb{C}$ of type $(K, \Phi)$. Let $h:\mathbb{S}\to GL(H^{1}(A, \mathbb{R}))$ be the $\mathbb{Q}$-Hodge structure. The  following statements hold: 
      \\\indent (1) The homomorphism 
      \begin{align}
         r_{h}^{-1}: {\rm Res}_{K^{\ast}/\mathbb{Q}}(\mathbb{G}_{m})\xrightarrow{{\rm Res}_{K^{\ast}/\mathbb{Q}}(\mu^{-1}_h)}{\rm Res}_{K^{\ast}/\mathbb{Q}}({\rm Res}_{K/\mathbb{Q}}(\mathbb{G}_{m})_{K^{\ast}})\xrightarrow{{\rm Nm}_{K^{\ast}/\mathbb{Q}}}{\rm Res}_{K/\mathbb{Q}}(\mathbb{G}_{m}) \notag 
      \end{align}is equal to the  \textbf{reflex norm} $N_{\Phi}$ of CM-type $(K, \Phi)$ (see \cite[2.1.3.3. Definition]{cco}).
      \\\indent (2)  For any $x\in K^{\ast\times}\subset \mathbb{C}^{\times}$,   we have  \begin{align}
        (\wedge^{2}\mu_{h})\otimes \mu_{\mathbb{Q}(1)}(x)=\mu_{h^{2}}(x), \notag 
        \end{align}
        where $\wedge^{2}\mu_{h}$ is a cocharacter defined by $x\mapsto \mu_{h}(x)\wedge\mu_{h}(x)\in \wedge^{2}H^{1}(A, \mathbb{C})$, $\mu_{h^{2}}$ is the cocharacter of the $\mathbb{Q}$-Hodge structure $H^{2}(A, \mathbb{Q})(1)$, and $\mu_{\mathbb{Q}(1)}$ is the cocharacter corresponding to the Tate twist $\mathbb{Q}(1)$. 
      \end{prop}
      \begin{proof}
      (1) 
      This follows from \cite[2.1.3.4. Proposition]{cco}. 
      \\(2) This follows from the existence of a canonical $\mathbb{Q}$-Hodge isometry $H^{2}(A, \mathbb{Q})(1)\cong \wedge^{2}H^{1}(A, \mathbb{Q})\otimes \mathbb{Q}(1)$. 
          
      \end{proof}
      \begin{prop}\label{simplecm} Let $(A, i)$ be a CM simple abelian surface over $\mathbb{C}$ of type $(K, \Phi)$, let $K^{\ast}$ be the reflex field of $K$, and let $\sigma\in {\rm Aut}(\mathbb{C}/K^{\ast})$. Let $s\in\mathbb{A}^{\times}_{K^{\ast}}$ be an id\'ele such that ${\rm art}_{K^{\ast}}(s)=\sigma|_{K^{\ast ab}}$. Then there exists a $K$-linear isogeny $\alpha: A^{\sigma}\to A$ such that the following diagram commutes:
          \begin{align}
              \xymatrix{T(A)_{\mathbb{A}_{f}}\ar[r]^{\alpha^{\ast}}& T(A^{\sigma})_{\mathbb{A}_{f}}\\ T(A)_{\mathbb{A}_{f}},\ar[u]^{\overline{s_{f}}/s_{f}}\ar[ur]_{\sigma^{\ast}}} \notag 
          \end{align}
           where $s_{f}\in \mathbb{A}^{\times}_{K^{\ast}, f}$ is the finite part of $s$. 
      \end{prop}
      \begin{proof} For a prime $l$, let $T_{l}A$ be the $l$-adic Tate module,  let $T_{f}A:=\prod_{l: prime}T_{l}A$, and let $V_{f}A:=T_{f}A\otimes\mathbb{Q}$.
          By \cite[Theorem 11.2]{milne2}, there exists a $K$-linear isogeny $\beta:  A\to A^{\sigma}$ such that the following diagram commutes: \begin{align}\xymatrix{ V_{f}A\ar[r]^{\beta}&V_{f}A^{\sigma}\\ V_{f}A, \ar[ru]_{\sigma^{-1}}\ar[u]^{N_{\Phi}(s)}&}\notag 
          \end{align} where $\sigma^{-1}: A\to A^{\sigma}$ is the inverse of the conjugation $A^{\sigma}\to A$.  Taking  duals, we obtain the commutative diagram
          \[\xymatrix{H^{1}(A, \mathbb{A}_{f})\ar[d]_{N_{\Phi}(s)}& H^{1}(A^{\sigma}, \mathbb{A}_{f})\ar[l]_{\beta^{\ast}}\ar[dl]^{\sigma^{-1\ast}}\\ H^{1}(A, \mathbb{A}_{f}).&}\]Let $\alpha: A^{\sigma}\to A$ be the inverse isogeny of $\beta$. By Proposition \ref{rh} (1), the diagram 
           \[\xymatrix{H^{1}(A, \mathbb{A}_{f})\ar[r]^{\alpha^{\ast}}& H^{1}(A^{\sigma}, \mathbb{A}_{f})\\ H^{1}(A, \mathbb{A}_{f}).\ar[u]^{r_{h}(s)}\ar[ru]_{\sigma^{\ast}}&}\]
           is commutative. 
         By Proposition \ref{rh} (2),  we obtain the following commutative diagram:
          \begin{align}
              \xymatrix{H^{2}(A, \mathbb{A}_{f})(1) \ar[r]^{\alpha^{\ast}}& H^{2}(A^{\sigma}, \mathbb{A}_{f})(1)\\ H^{2}(A, \mathbb{A}_{f})(1),\ar[u]^{r_{h^{2}}(s)}\ar[ru]_{\sigma^{\ast}}&} \notag 
          \end{align}
          where $r_{h^{2}}$ is defined as the composition
          \begin{align}
               {\rm Res}_{K^{\ast}/\mathbb{Q}}(\mathbb{G}_{m})\xrightarrow{{\rm Res}_{K^{\ast}/\mathbb{Q}}(\mu_{h^{2}})}{\rm Res}_{K^{\ast}/\mathbb{Q}}({\rm Res}_{K/\mathbb{Q}}(\mathbb{G}_{m})_{K^{\ast}})\xrightarrow{{\rm Nm}_{K^{\ast}/\mathbb{Q}}}{\rm Res}_{K/\mathbb{Q}}(\mathbb{G}_{m}). \notag
          \end{align}
          Since $r_{h^{2}}(s)$ acts trivially on ${\rm NS}(A)_{\mathbb{Q}}$, it induces an  action on $T(A)_{\mathbb{A}_{f}}$ via $\overline{s_{f}}/s_{f}$. Hence we obtain the required diagram. 
      \end{proof}
      \begin{cor}
          Let $(A, i)$ be a simple CM  abelian surface over a subfield $k\subset\mathbb{C}$ of CM-type $(K, \Phi)$. Let $G$ be the Mumford-Tate group of $T(A_{\mathbb{C}})_{\mathbb{Q}}$. Then the Galois representation $\rho:{\rm Gal}(\overline{k}/k)\to GL(T(A_{\mathbb{C}})_{\mathbb{Q}})(\mathbb{A}_{f})$ factors through  a continuous homomorphism
          \begin{align}
              \rho^{ab}: {\rm Gal}(k^{ab}/k)\to G(\mathbb{A}_{f})
         \notag  \end{align}
         such that the following diagram commutes:
         \begin{align}
            \xymatrix{ \mathbb{A}^{\times}_{k}\ar[d]_{{\rm Nm}_{k/K^{\ast}}}\ar[rr]^{{\rm art}_{k}}\ar[d]&&{\rm Gal}(k^{ab}/k)\ar[r]^{\rho^{ab}}& G(\mathbb{A}_{f})\ar[d]
          \\ \mathbb{A}^{\times}_{K^{\ast}}\ar[rr]_{proj}&&\mathbb{A}^{\times}_{K^{\ast}, f}\ar[r]_{\overline{(-)}/(-)}&G(\mathbb{A}_{f})/G(\mathbb{Q}).} \notag
         \end{align} 
      \end{cor}
      \begin{proof}Our first step is to construct a continuous homomorphism $\rho: {\rm Gal}(k^{ab}/k)\to G(\mathbb{A}_{f})$. Let $\sigma\in{\rm Gal}(\overline{k}/k)$. Choose an automorphism $\sigma^{'}\in {\rm Aut}(\mathbb{C}/K^{\ast})$ such that  $\sigma^{'}|_{\overline{k}}=\sigma$. Let $s\in \mathbb{A}^{\times}_{k}$ be an id\'ele such that ${\rm art}_{k}(s)=\sigma|_{k^{ab}}$. Then we have ${\rm art}_{K^{\ast}}({\rm Nm}_{k/K^{\ast}}(s))=\sigma|_{K^{\ast ab}}$. By Proposition \ref{simplecm}, there exists a $K$-isogeny $\alpha:A_{\mathbb{C}}\to A_{\mathbb{C}}$ such that \begin{align}\label{simplerel}\rho(\sigma)=\overline{{\rm Nm}_{k/K^{\ast}}(s)}/{\rm Nm}_{k/K^{\ast}}(s)\cdot\alpha^{\ast}\in G(\mathbb{A}_{f}).\tag{4.3}\end{align} 
      Since the group $G(\mathbb{A}_{f})$ is commutative, the Galois representation $\rho$ is abelian; hence  it induces a continuous homomorphism $\rho^{ab}:{\rm Gal}(k^{ab}/k)\to G(\mathbb{A}_{f})$. Since $\alpha^{\ast}\in G(\mathbb{Q})$, it follows from (\ref{simplerel}) that the diagram is commutative. 
      \end{proof}
      The following corollary is proved along the same vein as Theorem \ref{main theorem}.
      \begin{cor}\label{simple cm1} Let $(A, i)$ be a CM simple abelian surface of type $(K, \Phi)$. Then there exists a unique locally constant homomorphism $u^{'}: \mathbb{A}^{\times}_{k}\to K^{\ast\times }$ satisfying the following properties:
          \\\indent (i) Fix an embedding $\tau: K^{\ast}\to \mathbb{C}$. The function
          \begin{align}
              \chi^{'}_{\tau}: \mathbb{A}^{\times}_{k}\xrightarrow{u^{'}(-)\cdot \overline{{\rm Nm}_{k/K^{\ast}}(-)}/{\rm Nm}_{k/K^{\ast}}(-)}\mathbb{A}^{\times}_{K^{\ast}} \xrightarrow{\tau-{\rm projection}}\mathbb{C}^{\times}\notag 
          \end{align}
          is an algebraic Hecke character.
          \\\indent (ii) Let 
          \begin{align}
              \rho_{l}: {\rm Gal}(\overline{k}/k)\to GL(T_{l}(A_{\overline{k}})) \notag 
          \end{align}
          be the $l$-adic representation. Then $\rho_{l}$ is diagonalized by the algebraic Hecke characters $\chi_{\tau}$, where $\tau$ runs over the set ${\rm Hom}(K^{\ast}, \mathbb{C})$.
      \end{cor}
      \begin{prop}\label{onegai} Let $(A, i)$ be a CM simple abelian surface of type $(K, \Phi)$. Let $X:={\rm Km}(A)$ be the Kummer surface, and let $\chi_{\tau}$ be the algebraic Hecke characters in Theorem \ref{main theorem} for $X$. 
          Then, for every embedding $\tau\in {\rm Hom}(K^{\ast}, \mathbb{C})$, we have  $\chi_{\tau}=\chi^{'}_{\tau}$. 
      \end{prop}
      \begin{proof}By the construction of Kummer surfaces, there is an isomorphism of ${\rm Gal}(\overline{k}/k)$-modules
      \begin{align}
          H^{2}(X_{\overline{k}}, \mathbb{Q}_{l})(1)\cong H^{2}(A_{\overline{k}}, \mathbb{Q}_{l})(1)\oplus V, \notag 
      \end{align}
      where $V$ is a 16-dimensional $\mathbb{Q}_{l}$-vector space and is generated by the cohomology classes of the exceptional divisors of the minimal resolution $X\to A/\langle-1\rangle$. In particular, this induces an isomorphism of ${\rm Gal}(\overline{k}/k)$-modules
          \begin{align}
              T(A_{\mathbb{C}})_{\mathbb{Q}_{l}}\cong T(X_{\mathbb{C}})_{\mathbb{Q}_{l}}.  \notag 
          \end{align}
          For any $\sigma\in{\rm Gal}(\overline{k}/k)$, let $s\in \mathbb{A}^{\times}_{k}$ be an id\'ele such that  $\sigma|_{k^{ab}}={\rm art}_{k}(s)$.  By Proposition \ref{hecke1}, the action of $\sigma$ on $T(X_{\mathbb{C}})_{\mathbb{Q}_{l}}$ is given by 
          \begin{align}u(t)\overline{{\rm Nm}_{k/K^{\ast}}(t)}/{\rm Nm}_{k/K^{\ast}}(t).\notag 
          \end{align}On the other hand, by  Proposition \ref{simple cm1}, the action of $\sigma$ on $T(A)_{\mathbb{Q}_{l}}$ is given by 
          \begin{align}
          u^{'}(t)\overline{{\rm Nm}_{k/K^{\ast}}(t)}/{\rm Nm}_{k/K^{\ast}}(t). \notag 
          \end{align}Hence we have
          \begin{align}
              \chi_{\tau}(t)=\left(u(t)\frac{\overline{{\rm Nm}_{k/K^{\ast}}(t)}}{{\rm Nm}_{k/K^{\ast}}(t)}\right)_{\tau}=\left(u^{'}(t)\frac{\overline{{\rm Nm}_{k/K^{\ast}}(t)}}{{\rm Nm}_{k/K^{\ast}}(t)}\right)_{\tau}=\chi^{'}_{\tau}(t).\notag 
          \end{align}
      \end{proof} 
    
\end{document}